\documentclass[11pt]{amsart}
\usepackage{amsfonts}
\usepackage{amssymb}
\usepackage{graphicx,color}

\usepackage{setspace}
\usepackage{amscd}
\usepackage{mathrsfs}
\newtheorem{theorem}{Theorem}[section]

\newtheorem{corollary}[theorem]{Corollary}
\newtheorem{definition}[theorem]{Definition}

\numberwithin{equation}{section}

\def\DO{\mathcal D}
\def\RE{\mathbb R}
\def\CO{{\mathbb C}}
\def\N{\mathbb N}

\def\C{\mathcal C}

\def\A{\mathcal A}

\begin{document}

\title[Euler--Bernoulli beam equation with singular coefficients]
{Vibration modes of the Euler--Bernoulli beam equation with singularities}

\author{Nuno Costa Dias}
\author{Cristina Jorge}
\author{Jo\~{a}o Nuno Prata}

 \maketitle

\begin{abstract}
We consider the time dependent Euler--Bernoulli beam equa-tion with discontinuous and singular coefficients. Using an extension of the H\"ormander product of distributions with non-intersecting singular supports [L. H\"ormander, The Analysis of Linear Partial Diffe\-rential Operators I, Springer-Verlag, 1983], we obtain an explicit formulation of the differential problem which is strictly defined within the space of Schwartz distributions. We determine the general structure of its separable solutions and prove existence, uniqueness and regularity results under quite general conditions. This formalism is used to study the dynamics of an Euler--Bernoulli beam model with discontinuous flexural stiffness and structural cracks. We consider the cases of simply supported and clamped-clamped boundary conditions and study the relation between the characteristic frequencies of the beam and the position, magnitude and structure of the singularities in the flexural stiffness. Our results are compared with some recent formulations of the same problem.

\end{abstract}


\section{Introduction}

In this paper we consider the dynamic Euler--Bernoulli beam equation:
\begin{equation}\label{op}
 \partial_x^2 \left[ {a(x)  \partial_x^2 u(x,t)}\right] + b(x)\partial_t^2{u}(x,t) = 0, \quad x,\, t \in \RE 
\end{equation}
for coefficients $a$ and $b$ in the space of distributions $\A=\cup_{i=0}^{\infty} \partial_x^i[\C_p^\infty] \subset \DO'$, where $C_p^\infty$ is the space of piecewise smooth functions, $\DO'$ is the space of Schwartz distributions on $\RE$, and $\partial_x$, $\partial_t$ are distributional derivatives.

For general $a,b \in \A$, the equation (\ref{op}) is not well-defined if $u$ is non-smooth (because the terms on the left-hand side of \eqref{op} may involve products of non-regular distributions). In most cases, it does not display (non-trivial) smooth solutions either. Therefore, for $a,b \in \A$,  the classical formulation of (\ref{op}) is of limited interest.


 
 There are several formalisms that can be used to extend the domain of the differential expression (\ref{op}). Three interesting possibilities are the formulations in terms of Colombeau generalized functions \cite{Col84,CLNP89,HO09,Hor13}, the related approaches based on the sequential/very weak solutions formalism \cite{Blommaert,Garetto}, and the {\it intrinsic} formulations \cite{Cad09,Cad12,GRO16,HO07,Lep77}{\footnote{Hilbert space methods have also been widely used to study similar problems for Schr\"odinger operators with point interactions \cite{Albeverio1,Albeverio2,Dias3,Golovaty}.}}. The latter is defined within the space of Schwartz distributions, usually in terms of a multiplicative product defined in some subset of $\DO'$. 
 
 While the Colombeau and some sequential solutions formulations are more general, the intrinsic approaches are simpler to implement, but usually limited to particular cases. For instance, model products, which are the most general products in the hierarchy given by M. Oberguggenberger in [section 7, \cite{Obe92}],
have been used to extend the domain of ordinary differential equations (ODE)s with discontinuous coefficients \cite{HO07}. The resulting formula--tion is compatible with some
non-smooth solutions, but it is not well-defined in general for discontinuous ones.
  
In this paper we construct a new {\it intrinsic} formulation of (\ref{op}) by using the multiplicative product $*$ that was defined in \cite{DP08,DP09}. This product is an extension to $\A$ of the H\"ormander product of distributions with disjoint singular supports {[pag.55, \cite{Hor83}]}, and displays almost optimal properties. It is an associative, inner operation that extends the usual product of piecewise smooth functions, and satisfies the Leibniz rule with respect to the distribu--tional derivative. In fact, it is (essentially) the {\it unique} multiplicative product of distributions defined in some space $B \supseteq \A$ that satisfies all the conditions in the Schwartz impossibility result (cf. \cite{Sch54}) except for $B \supseteq \DO'$ \cite{DJP}.

The product $*$ has been used to obtain intrinsic formulations for a variety of differential problems with coefficients in $\A$ \cite{DPJ16,DJP19,DJP20,Dias3}. In general, it allows us to construct extensions of the original differential expressions which are well-defined in $\A$. For the case of (\ref{op}), we substitute the usual pointwise product functions by $*$ and obtain:
\begin{equation}\label{21}
\left[{a_0(x)* u''(x,t)+ u''(x,t)*a_1(x)}\right]'' + b_0(x) * \ddot{u}(x,t)+\ddot{u}(x,t)*b_1(x)=0,
\end{equation}
where $x,t \in \RE$, $a_0,a_1,b_0,b_1 \in \A$ and, as usual, $u''=\partial_x^2 u$,  $\ddot{u}=\partial_t^2 u$. 
The new equation (\ref{21}) has twice the coefficients of (\ref{op}) because the product $*$ is non-commutative. If $a_i,b_i$, $i=0,1$ are smooth then (\ref{21}) reduces to (\ref{op}) with $a=a_0+a_1$ and $b=b_0+b_1$. Hence, (\ref{21}) is an extension of (\ref{op}) that is well-defined for the case where both the coefficients and the prospective solutions are distributions in $\A$ (which includes the cases of piecewise smooth functions, the Dirac delta and all its derivatives). 
 
The formulation (\ref{21}) displays the following key features. First it is based on a sound and self-contained mathematical structure (an associative differential algebra of distributions where differential equations can be defined). Second it is quite general, allowing for a unified treatment of many different beam models. Lastly, it can be reformulated in terms of a    
 "{\it limit differential equation}". 
This result, proved in \cite{DJP20} for a similar case, states that
(\ref{21}) is equivalent to (all) the equations of the form 
\begin{equation}\label{LDE1}
\lim_{\epsilon \to 0^+}\left[\left(A_\epsilon(x) u''(x,t) \right)''+ B_\epsilon(x) \ddot{u}(x,t)\right]=0
\end{equation}
where the limit is taken in the distributional sense, and 
$$
A_\epsilon(x)=a_{0\epsilon}(x-\epsilon) + a_{1\epsilon}(x+\epsilon) \quad , \quad B_\epsilon(x)=b_{0\epsilon}(x-\epsilon) + b_{1\epsilon}(x+\epsilon)~,
$$
where $a_{i\epsilon},b_{i\epsilon}$, $i=0,1$, $\epsilon >0$ belong to a (quite general) class of one-parameter families of smooth functions with distributional limits $a_{i\epsilon} \to a_i$ and $b_{i\epsilon} \to b_i$ as $\epsilon \to 0^+$ (cf. Corollary 3.2, \cite{DJP20}). The reformulation (\ref{LDE1}) of (\ref{21}) is useful to clarify the role played by the product $*$ and by the "left" and "right" coefficients ($a_0,b_0$ and $a_1,b_1$, respectively) in (\ref{21}). We also note that (\ref{LDE1}) displays an interesting resemblance, but it is not equivalent to, the sequential/very weak solution formulations of \cite{Blommaert,Garetto}.

In the first part of the paper (section 3) we discuss the general properties of (\ref{21}). 
We prove existence and uniqueness results for its separable solutions, and study the conditions for which they are regular and singular. 
We also determine the general structure of these solutions, including the explicit form of the interface conditions at the points where they are non-smooth. 

\vspace{0.3cm}

The previous formalism can be used to study the dynamics of Euler--Bernoulli beam (EBB) models with discontinuous flexural stiffness and/or concentrated cracks (see \cite{Ata97} for the mechanical background). The intrinsic formulations of these systems are typically beset by some interesting mathe--matical problems, mainly related to the existence of singularities (cf. 
\cite{BC07,Cad09,Dim96,Gla04,HO07,YS01,YSR01}). In section 4 we use the general results of section 3 to study the following (formal) EBB model:
\begin{equation}\label{o}
 \left[ {a(x)  u''(x,t)}\right]'' + m\ddot{u}(x,t) = 0,\quad x \in \left[ {0,1}  \right], \quad t \geq 0
\end{equation}
satisfying pinned-pinned (PP) boundary conditions:
\begin{equation}\label{nnn}
u(0,t)=u''(0,t)=0 ~, ~u(1,t)=u''(1,t)=0 , \quad t \geq 0
\end{equation}
or clamped-clamped (CC) boundary conditions:
\begin{equation}\label{CCbc}
u(0,t)=u'(0,t)=0 ~ , ~ u(1,t)=u'(1,t)=0 , \quad t \geq 0 \,.
\end{equation}
Here:
\begin{itemize}

\item $u(x,t)$ is the transverse displacement of the beam (at the point $x$ and time $t$), 

\item $a(x)$ denotes the flexural stiffness, which is given
by $a=EI$ where $E$ is the modulus of elasticity and $I$ is the moment of inertia. We consider $a$ to be of the form:
\begin{equation} \label{lli}
a(x)  =  A( H (\xi_0-x) -\lambda_0 \delta_{\xi_0})+ 
 kA( H (x-\xi_0) -\lambda_1\delta_{\xi_0}) 
\end{equation}
where $H$ is the Heaviside step function, $A>0$ and $kA>0$ are the flexural stiffness constants in the sections $\left[0,\xi_0\right)$ and $\left(\xi_0,1\right]$, respectively ($\xi_0$ is the junction point; if $k=1$ the beam is homogeneous). Following \cite{Cad09}, the concentrated cracks are modeled by Dirac deltas in the flexural stiffness. Since we are considering the (more) general case where the crack and the junction point may be located at the same position, the flexural stiffness displays two parameters, $\lambda_0$ and $\lambda_1$,  which are related to the intensity of the crack at the left and right sides of the junction point $\xi_0$. 

\item $m$ is the mass per unit length of the beam, which we assume to be constant.
\end{itemize}

The formal model (\ref{o},\,\ref{lli}) describes a beam with a junction point (if $k \not=1$) and/or a structural crack (if $\lambda_0\not=0$ or $\lambda_1\not=0$) both located at $\xi_0$. 
In section 4 we re-write (\ref{o},\,\ref{lli}) in the form (\ref{21}), and obtain an explicit formulation  of the EBB equation which is well-defined for arbitrary coefficients $a_0,a_1$ of the form (\ref{lli}), and reduces to (\ref{o}) exactly if the coefficients are smooth.

 Using this formulation, we investigate the effect of the junction points and concentrated cracks on the characteristic frequencies of the EBB (a problem with important applications in engineering \cite{Dim96,Gla04,Mor93,Sci54}). We also determine the general structure of the vibration modes, including the specific interface conditions at the junction point, and present numerical results for the relation between the characteristic frequencies and the crack position, intensity and structure. Notably, our numerical results are quite similar to those obtained in \cite{Cad09} using a different formalism, and which are consistent with some experimental data (cf. \cite{Cad09}).
    
To conclude, let us point out that the formalism presented in this paper can be easily used/extended to study many other beam models. These include the cases of arbitrary finite number of cracks/junction points, external distributional forces, and/or distributional mass densities. We can also consider the cases of more general distributional flexural stiffness (e.g. displa--ying higher order derivatives of the Dirac delta), other types of beams (e.g. Timoshenko, Rayleigh, Vlasov) and other  boundary conditions (e.g. classic - simply clamped, supported, sliding; and non-classical - elastic, inertial, dissipative) \cite{Jun12, Kla20,Mau93,Sil16,Vaz16}.

 \vspace{0.3cm}

\noindent\textbf{Notation}. Let $\Omega$ and ${\overline\Omega}$ denote an arbitrary open interval of $\RE$ and its closure, respectively. The functional spaces are denoted by calligraphic capital letters, ($\A(\Omega)$, $\C^\infty(\Omega)$, $\DO(\Omega)$,....). If $\Omega = \RE$ we write only $\A$, $\C^\infty$, $\DO$,... In general, we do not distinguish a locally integrable function from the associated regular distribution (the only exception is in Definition \ref{78}, where we write $\phi_{\DO'}$ to denote the regular distribution associated to the smooth function $\phi$). As usual, $H$ is the Heaviside step function, $\chi_{\Omega}$ is the characteristic function of $\Omega$, and $\delta_{x_0}$ is the Dirac measure with support at $x_0$. If $x_0=0$ we write only $\delta$. The derivatives are always considered in the distributional sense. We use the standard notations: $u'(x,t)=\partial_x u(x,t)$, $\dot{u}(x,t)=\partial_t u(x,t)$ and $\psi^{(n)}(x)=\partial_x^n\psi(x)$.

\section {Schwartz Distributions and ODEs with distributional coefficients}

In this section we review some basic notions about Schwartz distributions and present the main properties of the multiplicative product $*$ that was proposed in \cite{DP09}. We also review the formulation and some relevant properties of the ODEs of the form 
\begin{equation} \label{eq11}
\sum\limits_{i = 0}^n  \left(a_i(x) * \psi^{(i)} (x)+  \psi^{(i)} (x)* b_i(x)\right)= f(x)
\end{equation}
where $a_i, b_i \in \A$, $f \in \C^{\infty}$. The equations of this form will be important for our formulation of the EBB equation. For details and proofs the reader should refer to \cite {DJP20}.  

\subsection {A multiplicative product of Schwartz distributions}

The space of test functions $\DO(\Omega)$ is the set of all functions which are smooth and have support on a compact subset of $\Omega$. Let $\DO'(\Omega)$ be its dual, the space of Schwartz distributions. If $\Omega =\RE$, we write simply $\DO'$. Let $F|_{\Omega}$ denote the restriction of $F \in\DO'$ to the space $\DO(\Omega)$. We have, of course, $F|_{\Omega} \in \DO'(\Omega)$. 

The support of $F \in \DO'$, denoted by supp $F$, is defined as the complement of the largest open set where $F$ vanishes, and the singular support of $F \in \DO'$, denoted by sing supp $F$, is the complement of the largest open set where $F$ is a smooth function.

 The distribution $F \in \DO'$ is of order $n$ (and we write $n=$ ord $F$) iff $F$ is the $n$th order
distributional derivative (but not a lower order distributional
derivative) of a regular distribution  [pag.43, \cite{Kan98}]. A distribution of order
zero is a regular distribution. 

Finally, let $\C_p^{\infty}$ be the space of piecewise smooth functions on $\RE$, i.e. the space of functions which are smooth, except on a finite set of points, where they and all their derivatives have finite left and right limits. 

A distributional extension of the space $\C_p^{\infty}$ is given by:

\begin{definition}
Let $\A$ be the space of all functions in $\C_p^{\infty}$ - regarded as Schwartz distributions - together with
all their distributional derivatives to all orders. Moreover, for $\Omega \subset \RE$ an open set, the space of
distributions of the form $F|_{\Omega}$, where $F \in \A$, is denoted by $\A(\Omega)$.
\end{definition}

Notice that $\C_p^{\infty} \subset \A \subset \DO'$. The elements of $\A$ are distributions with finite singular
support. They can be written in the form $F=\Delta_F +f$, where $\Delta_F$ is a distribution with finite support
(and thus a finite linear combination of Dirac deltas and their derivatives) and $f \in \C_p^{\infty}$.  The explicit form of a distribution in $\A$ is given in the following

\begin{theorem}
$F \in \A$ iff there is a finite set $I=\{x_1,...,x_m\}$ (where $x_i<x_k$ for $i<k$) associated with a set of
open intervals $\Omega_i=(x_i,x_{i+1})$, $i=0,..,m$ (where $x_0=-\infty$ and $x_{m+1}=+\infty$) such that
($\chi_{\Omega_i}$ is the characteristic function of $\Omega_i$):
\begin{equation}\label{FormF}
F= \sum_{i=1}^m \sum_{j=0}^n c_{ij}\delta^{(j)}_{x_i}+ \sum_{i=0}^m f_i \chi_{\Omega_i}
\end{equation}
for some $c_{ij} \in \CO$ and $f_i \in \C^{\infty}$. The singular support of $F$ is, of course, a subset of $I$.
\end{theorem}

The dual product of a distribution $F \in \DO'(\Xi)$ by a smooth function $g \in
\C^\infty(\Xi)$, where $\Xi\subseteq \RE$, is defined by:
\begin{equation} \label{dual}
\langle F g, t \rangle= \langle F, gt \rangle \quad , \quad \forall \,t \in \DO (\Xi)\, .
\end{equation}

The H\"ormander product of distributions with non-intersecting singular supports, introduced by L. H\"ormander in (pag.55, \cite{Hor83}), generalizes the dual product, and admits the following restriction to the space $\A$:

\begin{definition}
Let $F,G \in \A$  be two distributions with disjoint singular supports. Then there exists a finite open cover of
$\RE$ (denote it by $\{\Xi_i \subset \RE,\, i=1,..,d \}$) such that, on each open set $\Xi_i$, either $F$ or $G$
is a smooth function. The H\"ormander product of $F$ by $G$ is then defined as the unique distribution $F
\cdot G \in \A$ that satisfies:
$$
\left(F \cdot G\right)|_{\Xi_i}= F|_{\Xi_i} G|_{\Xi_i} \quad , \quad  i=1,..,d.$$
where $F|_{\Xi_i} G|_{\Xi_i}$ denotes the dual product of a distribution by a smooth function.
\end{definition}

Finally, the product $*$ extends the H\"ormander product to the case of an arbitrary pair of distributions in $\A $: 

\begin{definition}\label{2.4}
Let $F,G \in \A$.
The multiplicative product $*$ is defined by
\begin{equation} \label{prod}
F * G= \lim_{\epsilon \downarrow 0} F(x) \cdot G(x+\epsilon),
\end{equation}
where the product in $F(x) \cdot G(x+\epsilon)$ is the H\"ormander product (which is well-defined for sufficiently small $\epsilon >0$) and the limit is taken in the distributional sense.

If $F,G$ are generalized functions of the position $x$ and time $t$ (as in eq.(\ref{21})), such that $F(\cdot,t), G(\cdot,t) \in \A$ for all $t$, then (\ref{prod}) extends trivially:
$$
F * G= \lim_{\epsilon \downarrow 0} F(x,t) \cdot G(x+\epsilon,t)~.
$$
\end{definition}

The explicit form of $F*G$ is given by:

\begin{theorem}\label{2.5}
 Let $F,G \in \A$ and let $I_F$ and $I_G$ be the singular supports of $F$ and $G$, respectively. Let
$I=I_F \cup I_G=\{x_1 <...<x_m\}$. Let
$\Omega_i=(x_i,x_{i+1})$, $i=0,..,m$ (with $x_0=-\infty$ and $x_{m+1}=+\infty$). The distributions 
$F$ and $G$ can be written in the form (cf. Theorem 2.2):
\begin{eqnarray}\label{1}
F &=& \sum_{i=1}^m \sum_{j=0}^n a_{ij}\delta^{(j)}_{x_i} + \sum_{i=0}^m f_i \chi_{\Omega_i} \nonumber \\
G &=& \sum_{i=1}^m \sum_{j=0}^n b_{ij}\delta^{(j)}_{x_i} + \sum_{i=0}^m g_i \chi_{\Omega_i}
\end{eqnarray}
for some $n \in \N_0$ and $f_i,g_i \in \C^\infty$. Notice that $a_{ij}=0$ if $x_i \notin I_F$ or if $j \ge$ ord $F$, and likewise for
$b_{ij}$.
In this notation $F*G$ is given explicitly by
\begin{equation} \label{prodf}
F * G = \sum_{i=1}^m \sum_{j=0}^n \left[ a_{ij} g_i  + b_{ij} f_{i-1} \right] \cdot \delta^{(j)}_{x_i} +
\sum_{i=0}^m f_i g_i \chi_{\Omega_i}.
\end{equation}
and $F*G \in \A$.
\end{theorem}

Some simple results that follow from (\ref{prodf}) are ($x_0,x_1 \in \RE$ and $i,j \in \N_0$):
$$
H(x_0-x) * \delta^{(i)}_{x_0}= \delta^{(i)}_{x_0}* H(x-x_0)=\delta^{(i)}_{x_0}
$$
\begin{equation} \label{prods}
H(x-x_0) * \delta^{(i)}_{x_0}= \delta^{(i)}_{x_0}* H(x_0-x)=0 ~.
\end{equation}
$$
\delta^{(i)}_{x_0} *\delta^{(j)}_{x_1}=0
$$
The main properties of $*$ are summarized in the following Theorem:
\begin{theorem}
The product $*$ is an inner operation, it is associative, distribu--tive, non-commutative and it
reproduces the H\"ormander product for distribu--tions with disjoint singular supports. The distributional derivative satisfies the Leibniz rule with respect to the product $*$.
\end{theorem}

We conclude that the space $\A$, endowed with the product $*$, is
an associative (but noncommutative) differential algebra of
distributions. In fact, it is the (essentially) unique algebra that satisfies all the properties in the Schwartz impossibility result \cite{DJP,Sch54} with the exception of $\A \supseteq \DO'$ (we have, instead, $\A \subsetneq \DO'$).

\subsection {Linear differential equations with distributional coefficients}

Consider the ODE (\ref{eq11}) with coefficients $a_i,b_i \in \A$, $f \in \C^{\infty}$ and boundary conditions: 
\begin{equation}\label{IC}
\overline{\psi(x_0)} =\overline{C}
\end{equation}
where $x_0$ is a {\it regular} point of (\ref{eq11}) (see below), and
$$
\overline{\psi(x_0)} =(\psi(x_0),...,\psi^{(n-1)}(x_0))^T \quad , \quad \overline{C}=(C_1,...,C_n)^T \in \CO^n \, 
$$
where the superscript T denotes transposition.

The set
$$
I = \cup_{i=0}^n ~\left({\rm sing ~supp}~ a_i  \cup ~{\rm sing ~supp} ~b_i \right)
$$
is called the {\it singular set} of (\ref{eq11}).
A point $x_0 \in I$ is a {\it singular point} of (\ref{eq11}); otherwise it is a {\it regular point}. An interval is a {\it regular interval} of (\ref{eq11}) iff all its points are regular. 

\begin{definition}:\label{78}

\begin{itemize}
    \item [(A1)]  $\psi$ is a solution of the ODE (\ref{eq11}) iff $\psi \in \A$ and $\psi$ satisfies (\ref{eq11}) in the distributional sense.

   \item [(A2)]  $\psi$ satisfies the boundary conditions (\ref{IC}) at a regular point $x_0$ iff there is a open interval $\Omega \ni x_0$, and a function $\phi \in \C^{\infty}$ such that (i) $\psi=\phi_{\DO'}$ on $\Omega$ and (ii) $\phi^{(i)}(x_0)= C_{i+1}$, $i=0,...,n-1$.

   \item [(A3)] If $\psi$  satisfies (A1) and (A2), then $\psi$ is a solution of the boundary value problem (BVP) (\ref{eq11},\ref{IC}).
\end{itemize}
\end{definition}

We will assume that  (\ref{eq11}) satisfies the following property:

\begin{definition} Sectionally Regular ODE. \label{Re}

The ODE (\ref{eq11}) is said to be sectionally regular iff for every open regular interval $\Omega$ of
(\ref{eq11}), arbitrary $x_0\in \overline{\Omega}$ and $\overline{C}\in \CO^n$, the BVP (let $a_{i{\overline\Omega} },b_{i{\overline\Omega} }$ denote $\C^{\infty}(\overline{\Omega})$-extensions of the restrictions $a_{i}|_\Omega,b_{i}|_\Omega$)
\begin{equation}\label{CODE}
\sum\limits_{i = 0}^n (a_{i{\overline\Omega}}+b_{i{\overline\Omega}}) \psi^{(i)}_{\overline\Omega}= f|_{\overline{\Omega}} \quad , \quad
\overline{\psi_{\overline\Omega} (x_0)}=\overline{C}
\end{equation}
has a unique solution $\psi_{\overline\Omega} \in \C^{\infty}(\overline{\Omega})$.
\end{definition}

The next Theorem provides sufficient conditions for an ODE to be section--ally regular.

\begin{theorem}\label{SecReg}
Consider the ODE (\ref{eq11}) with coefficients $a_i,b_i \in  \A$ , $i = 0,...,n$ such that,
    for every open regular interval $\Omega$,
$$
a_{n\overline\Omega}(x) + b_{n\overline\Omega}(x) \not= 0 \quad , \quad \forall x \in \overline\Omega.
$$
Then (\ref{eq11}) is sectionally regular. The proof is given in [Theorem 3.4, \cite{DJP20}].
\end{theorem}

The next Theorems \ref{1io} and \ref{1ii} describe the general structure of the solutions of (\ref{eq11}) and state the conditions for which these solutions are regular or singular distributions. Theorem \ref{1ii} also provides the interface conditions satisfied by the solutions of (\ref{eq11}) at the singular points. These results were obtained in \cite{DJP20} for the case of one singular point (cf. Theorems 3.6, 3.8 and 3.10 in \cite{DJP20}). Here we present a generalization of these results for the case of an arbitrary (finite) number of singular points. 
  
\begin{theorem}\label{1io}
Consider the $nth$-order ODE (\ref{eq11}) with coefficients $a_i,b_i \in \A ,~ ~ i=0,..,n$. Let $I=\{\xi_1< ...<\xi_m\}$ be the singular set of (\ref{eq11}), and assume that the leading coefficients $a_n$ and $b_n$ satisfy the following conditions:

\begin{itemize}

\item[(C1)] \hspace{0.3cm} $a_n(\xi^-) + b_n(\xi^+) \not=0, ~ \forall \xi \in I$

\item[(C2)] \hspace{0.25cm} For every regular open set $\Omega$,    
$
a_{n\overline{\Omega}}(x) + b_{n\overline{\Omega}}(x) \not=0, \, \forall x \in \overline{\Omega}~.
$  
\end{itemize}
Then the general solution of (\ref{eq11}) is of the form ($m=\sharp I$):
\begin{equation} \label{GF}
\psi=\sum_{k=0}^m \chi_{\Omega_k} \psi_k +\Delta
\end{equation}
where $\Omega_k=(\xi_k,\xi_{k+1}), ~k=0,..,m$ ($\xi_0=-\infty$ and $\xi_{m+1}=+\infty$),  $\chi_{\Omega}$ is the characteristic function of $\Omega$, and $\psi_k \in \C^{\infty}(\overline{\Omega_k})$ satisfies:
\begin{equation}\label{R-}
\sum_{i=0}^n(a_{i\overline{\Omega_k}}+b_{i\overline{\Omega_k}}) \psi^{(i)}_k =f \quad \mbox{on} \quad \overline{\Omega_k} ~.
\end{equation}
Moreover, let $M=$ max\,$\{$ord $a_i,~$ord $b_i:~i=0,..,n\}$. Then $\Delta \in \DO'$ satisfies:
\begin{enumerate}
	\item [(i)] if $M \le n$ then $\Delta=0$.
	\item [(ii)] if $M > n$ then supp $\Delta \subseteq I$ and ord $\Delta \le M-n$.
\end{enumerate}

\end{theorem}

\begin{proof}
This Theorem is a trivial generalization of  Theorem 3.6 in \cite{DJP20} for the case where $\sharp I >1$, and the proof follows exactly the same steps. 

Here, in addition, we do not assume that (\ref{eq11}) is sectionally regular (as we did in \cite{DJP20}) but this property follows easily from the condition (C2) by taking into account Theorem \ref{SecReg}.
\end{proof}

The next Theorem is also a straightforward generalization of the Theorems 3.8 and 3.10 (3),(6) of \cite{DJP20}, for the case $\sharp I >1$.
 
\begin{theorem}\label{1ii}
Consider the ODE (\ref{eq11}) satisfying the conditions of Theorem \ref{1io}. Let $I=\{\xi_1 <...<\xi_m\}$ be again the singular set of (\ref{eq11}), and let $\xi_0=-\infty$ and $\xi_{m+1}=+\infty$. 

If $M\le n$ then every solution $\psi$ is of the form (\ref{GF}) with $\Delta =0$.
Moreover, $\psi$ satisfies $m$ interface conditions (one at each singular point $\xi_k \in I$) of the form:
\begin{equation*}\label{Feq}
{\bf A}_k \overline{\psi_{k-1}(\xi_k)} = {\bf B}_k \overline{\psi_{k}(\xi_k)} ~,\quad k=1,...,m
\end{equation*}
where $
\overline{\psi_k(\xi)}=(\psi_k(\xi),...,\psi^{(n-1)}_k(\xi))^T$ and ${\bf A}_k,{\bf B}_k$ are $n \times n$ (in general complex valued) matrices.

Finally, for boundary conditions (\ref{IC}) given at  $x_0 \in (\xi_{k} , \xi_{k+1})$ for some $ k=0,..,m$, the solution of the BVP (\ref{eq11},\ref{IC}) exists and is unique if and only if $\det {\bf B}_s \not=0$ for all $s$ such that $k+1 \le s \le m$ (if $k \le m-1$), and $\det {\bf A}_r \not=0$ for all $r$ such that $1 \le r \le k$ (if $k \ge 1$).

\end{theorem}

\section{Generalized solutions of the Euler--Bernoulli beam equation with singular coefficients}\label{sec:3}

In this section we use the previous formalism to obtain an intrinsic formula--tion of the EBB equation (\ref{op}) for the general case $a,b\in\A$. We  study the properties of the separable solutions, provide the conditions for which they are regular or singular, and prove existence and uniqueness results. 

Following the general approach of section 2, our first step is to re-write the equation (\ref{op}) in the form (\ref{21}), which we reproduce here for future reference:
\begin{equation}\label{rrr}
\left[{a_0(x)* u''(x,t)+ u''(x,t)*a_1(x)}\right]'' + b_0(x)*\ddot{u}(x,t) + \ddot{u}(x,t)*b_1(x)=0,
\end{equation}
where $x,t \in \RE$.
This equation is well-defined for arbitrary $a_i,b_i \in \A$, $i=0,1$. Moreover, if the coefficients $a_i,b_i$ are smooth then (\ref{rrr}) reduces to (\ref{op}) with $a=a_0+a_1$ and $b=b_0+b_1$ (recall that the product $*$ of a smooth function by a distribution $F \in \A$ is commutative and reproduces the dual product).


The separable solutions of (\ref{rrr}) are of the form
\begin{equation*}
u(x,t)=\phi(x)T(t),
\end{equation*}
substituting these in (\ref{rrr}) we get the eigenvalue equations:
\begin{equation} \label{12i}
\ddot{T}(t)+w^2T(t)=0
\end{equation}
\begin{equation}\label{12ii}
\left[a_0(x)* \phi''(x) + \phi''(x) * a_1(x)\right]'' - w^2\left(b_0(x)*\phi(x)+\phi(x)*b_1(x)\right)=0
\end{equation}
where $w$ is the frequency associated with the vibration mode $\phi$.

The general solutions of (\ref{12i}) are of the form:
\begin{equation}
T(t)=P\cos(w\,t)+Q\sin(w\,t)
\end{equation}
where $P$ and $Q$ are integration constants.

On the other hand, the equation (\ref{12ii}) can be written more explicitly as 
\begin{equation} \label{12ik}
\sum_{i=0}^2 {2 \choose i} 
 \left(a^{(i)}_0* \phi^{(4-i)}+\phi^{(4-i)}*a^{(i)}_1 \right)-w^2\left(b_0*\phi+\phi*b_1\right)=0
\end{equation}

The next results concern the properties of the solutions of (\ref{12ik}). We start with a Corollary of Theorem \ref{1io}.

\begin{corollary}\label{1i}
Consider the  ODE (\ref{12ik}) and let $I$ be its singular set. Assu--me that the coefficients $a_0$ and $a_1$ satisfy the conditions:
\begin{itemize}

\item[(C1')] \hspace{0.3cm} $a_0(\xi^-) + a_1(\xi^+) \not=0, \quad \forall \xi \in I$

\item[(C2')] \hspace{0.25cm} For every regular open set $\Omega$,    
$
a_{0\overline{\Omega}}(x) + a_{1\overline{\Omega}}(x) \not=0, \quad \forall x \in \overline{\Omega}
$  
\end{itemize}
Let 
$$
M= {\rm max}\,\{ {\rm ord} \, a''_i,  \, {\rm ord} \, b_i, \, i=0,1 \}.
$$
Then:
\begin{enumerate}
	\item [(i)]If $M \le 4$ then the solutions of (\ref{12ik}) satisfy ord $\phi =0\Rightarrow \phi\in\C_p^{\infty}$. 
	\item [(ii)]If $M > 4$ then  ord $\phi \le M-4 $.
\end{enumerate}

\end{corollary}

\begin{proof}
The leading coefficients of (\ref{12ik}) are $a_0$ and $a_1$ (which are associated with $\phi^{(4)}$). These coefficients satisfy the conditions (C1') and (C2') which are equivalent to the conditions (C1) and (C2) in Theorem \ref{1io}. Moreover $M$ is the maximal order of the coefficients of (\ref{12ik}). Hence, the statements (i) and (ii) follow directly from Theorem \ref{1io} (and the fact that (\ref{12ik}) is of 4th-order).
\end{proof}

The next Theorem is the main result of this section. It considers the case $M\le 4$ in detail. We assume, to simplify the presentation, that there is only one singular point (the generalization for an arbitrary, but finite, number of singular points is straightforward). 
 Let the singular point be $\xi_0$, and let $M \le 4$.  Under these conditions the most general form of the coefficients $a_i,b_i \in \A$, $i=0,1$ is:
 \begin{eqnarray}\label{jjjk}
a_i & = & a_{i-}H(\xi_0-x)+ a_{i+}H(x-\xi_0)+\sum_{j=0}^1 A_{ij}\delta_{\xi_0}^{(j)} \\
\label{jjjkkk} b_i& = & b_{i-}H(\xi_0-x)+ b_{i+}H(x-\xi_0)+ \sum_{j=0}^3 B_{ij}\delta_{\xi_0}^{(j)}
\end{eqnarray}
where $a_{i\pm},b_{i\pm} \in \C^{\infty}$, $A_{ij},B_{ij} \in \CO$, $i=0,1$. The coefficients $a_{i-},b_{i-}$ and $a_{i+},b_{i+}$ could have been defined on $(-\infty,\xi_0)$ and $(\xi_0,+\infty)$, respectively. However, since they admit a smooth extension to $x=\xi_0$, and hence to $\RE$ (cf. Whitney's extension theorem, \cite{Fle77,Whi34}), we assume from the beginning that their domain is $\RE$.  
 	
 	We introduce for future reference the following notation ($i=0,1$):
\begin{equation}\label{notationa}
a=a_{0-}+a_{1+} \quad , \quad a_{\pm}=a_{0\pm}+a_{1\pm}
\end{equation}
$$
v_i=a_{i+}-a_{i-} \quad , \quad 
b_{\pm}=b_{0\pm}+b_{1\pm} \, .
$$
We then have:
\begin{theorem} \label{jjj}
Consider the  ODE (\ref{12ik}) with coefficients (\ref{jjjk},\ref{jjjkkk}) satisfying the conditions:
\begin{itemize}

\item[(C1'')] \hspace{0.3cm} $a(\xi_0) \not=0$

\item[(C2'')] \hspace{0.3cm}   
$a_{\pm}(x) \not=0, \quad \forall x \in \RE$. 
\end{itemize}
Then every solution of (\ref{12ik}) is of the form
\begin{equation}\label{Solution1}
\phi (x)=H (\xi_0-x) \phi_1 (x) + H (x-\xi_0)\phi_2 (x) 
\end{equation}
where $\phi_1 , \phi_2 \in \C^\infty$ satisfy the ODEs:
\begin{equation}\label{ODESReg}
\sum_{i=0}^{2} {2 \choose i}  a_-^{(i)}\phi^{(4-i)}_1-w^2b_-\phi_1=0 \text{ and }  \sum_{i=0}^{2} {2 \choose i}  a_+^{(i)}\phi^{(4-i)}_2-w^2b_+\phi_2=0 
\end{equation} 
on $(-\infty,\xi_0]$ and $[\xi_0,+\infty)$, respectively.

Moreover, at the non-regular point $\xi_0$, the solutions (\ref{Solution1}) satisfy the interface condition:
\begin{equation}\label{ICond}
{\bf A}(\xi_0) \overline{\phi_1(\xi_0)} = {\bf B}(\xi_0) \overline{\phi_2(\xi_0)}
\end{equation}
where ${\bf A}$ and ${\bf B}$ are the matrices (cf.(\ref{jjjk},~\ref{jjjkkk}, \ref{notationa}))
\begin{equation} \label{ee}
{\bf A}=\left( {\begin{array}{*{20}{c}}
  B_{10}w^2&-B_{11}w^2&a_-'+B_{12}w^2& a_--B_{13}w^2\\ 
 B_{11}w^2&-2B_{12}w^2&a_-+3B_{13}w^2&0\\ 
 -a'+B_{12}w^2&a-3B_{13}w^2&-A_{10}& A_{11} \\ 
a+B_{13}w^2&0&-A_{11}&0
\end{array}} \right) 
\end{equation}
\begin{equation} \label{eee}
{\bf B}=\left( {\begin{array}{*{20}{c}}
  -B_{00}w^2&B_{01}w^2&a_+'-B_{02}w^2& a_++B_{03}w^2\\ 
 -B_{01}w^2&2B_{02}w^2&a_+-3B_{03}w^2&0\\ 
 -a'-B_{02}w^2&a+B_{03}w^2&A_{00}&-A_{01}  \\ 
a-B_{03}w^2&0&A_{01}&0
\end{array}} \right) 
\end{equation}
and 
$$
\overline{\phi_i}= \left( 
  {{\phi_i}},\,{{\phi'_i}},\,{{\phi''_i}},\,{{\phi'''_i}} \right)^T, \, i=1,2
$$
\end{theorem}

\begin{proof}
The ODE (\ref{12ik}), with coefficients (\ref{jjjk},\,\ref{jjjkkk}) of the form (C1'') and (C2''), satisfies the prerequisites of Theorem \ref{1io} with $I=\{\xi_0\}$. Since, in this case, $\Omega_0=(-\infty, \xi_0)$ and 
$\Omega_1=(\xi_0,+\infty)$, we have:
$$
\chi_{\Omega_0}=H(\xi_0-x) \quad \mbox{and} \quad 
\chi_{\Omega_1}=H(x-\xi_0).
$$
Moreover $M\le 4$, and so it follows from Theorem \ref{1ii} that $\phi$ is of the form (\ref{Solution1}), and also (cf. Theorem \ref{1io}, eq.(\ref{R-})) that $\phi_1$ and $\phi_2$ satisfy (\ref{ODESReg}) on $\overline{\Omega_0}$ and $\overline{\Omega_1}$, respectively. 

In order to determine the explicit form of the interface condition let us substitute the coefficients (\ref{jjjk},\,\ref{jjjkkk}) into 
(\ref{12ik}); we get:
\begin{eqnarray}\label{88}
& &\sum_{i=0}^{2} {2 \choose i} \left[(a_{0-}^{(i)}\chi_{\Omega_0}+a_{0+}^{(i)}\chi_{\Omega_1} )*
 \phi^{(4-i)}+\phi^{(4-i)}*(a_{1-}^{(i)}\chi_{\Omega_0}+a_{1+}^{(i)}\chi_{\Omega_1})\right]\\
\nonumber&& -w^2 \left[(b_{0-}\chi_{\Omega_0}+b_{0+}\chi_{\Omega_1} )*\phi+\phi*(b_{1-}\chi_{\Omega_0}+b_{1+}\chi_{\Omega_1}\right] + B(x)=0
\end{eqnarray}
where $ B(x)$ includes all the terms displaying a Dirac delta or one of its derivatives, and is given by:
\begin{eqnarray}\label{B}
 && B(x)=\sum\limits_{i=0}^2 \sum\limits_{j=0}^1 {2 \choose i} \left[A_{0j} \delta_{\xi_0} ^{(j+i)}*\phi^{(4-i)}+\phi^{(4-i)}*A_{1j} \delta_{\xi_0} ^{(j+i)}\right]\\
&+& \sum\limits_{i = 1}^2 {2 \choose i} \left[(a_{0+}-a_{0-})\delta_{\xi_0} ^{(i-1)} *\phi^{(4-i)}+\phi^{(4-i)}*(a_{1+}-a_{1-})\delta_{\xi_0} ^{(i-1)}\right]\nonumber \\
&+& 2(a'_{0+}-a'_{0-})\delta_{\xi_0}*\phi''+ 2\phi''*(a'_{1+}-a'_{1-})\delta_{\xi_0}\nonumber\\
&-&w^2 \sum\limits_{j=0}^3\left[B_{0j}\delta_{\xi_0}^{(j)}*\phi+\phi* B_{1j}\delta_{\xi_0}^{(j)}\right]~. \nonumber
\end{eqnarray}
For an arbitrary $\phi \in \A$, we have supp $B(x) \subseteq \left\{\xi_0\right\}$. 
Substituting the solution (\ref{Solution1}) into (\ref{88},\,\ref{B}), and taking into account the relation ($i \ge 1$)
\begin{equation}
\phi^{(i)}=H(\xi_0-x)\phi_1^{(i)} +H(x-\xi_0) \phi_2^{(i)} +  \sum_{j=1}^i {i\choose j} \left(\phi_2- \phi_1 \right)^{{(i-j)}}\delta^{(j-1)}_{\xi_0} 
\end{equation}
and the equations (\ref{ODESReg}), we get (cf.(\ref{notationa})):
\begin{eqnarray} \label{Eq3.17}
&&\sum\limits_{i=0}^2 \sum\limits_{j=1}^{4-i}{2 \choose i}{4-i\choose j} a^{(i)} ( \phi_2-\phi_1)^{(4-i-j)}\delta^{(j-1)}_{\xi_0}\\
&+&\sum\limits_{m=0}^1\sum\limits_{i=0}^2 \sum\limits_{j=0}^{1}{2 \choose i} A_{mj}\phi_{2-m}^{(4-i)}\delta^{(j+i)}_{\xi_0}+\sum\limits_{m=0}^1 \sum\limits_{i=1}^{2}{2 \choose i} v_m\phi_{2-m}^{(4-i)}\delta^{(i-1)}_{\xi_0}\nonumber\\
&+&2\sum\limits_{m=0}^1 v_m'\phi''_{2-m}\delta_{\xi_0}
-w^2\sum\limits_{m=0}^1 \sum\limits_{j=0}^{3}B_{mj}\phi_{_{2-m}}\delta^{(j)}_{\xi_0}=0 ~. \nonumber
\end{eqnarray}
In this calculation we used (\ref{prods}), particularly the identity $\delta^{(i)}_{x_0} *\delta^{(j)}_{x_1}=0$, valid for all $x_0,x_1 \in \RE,~i,j\in \N_0$. 

To proceed we use the relation (cf.[page 38, \cite{Kan98}]): 
$$
f^{(s)}(x)\delta_{\xi_0}^{(n)}=(-1)^n\sum^n_{k=0} (-1)^k{n \choose k} f^{(s+n-k)} (\xi_0)\delta_{\xi_0}^{(k)}, \quad f \in \C^{s+n}, \quad s,n \in \N_0
$$
and rewrite (\ref{Eq3.17}) as a condition for the values of $\phi_1$ and $\phi_2$ (and their derivatives) at $x=\xi_0$:
$$
\begin{array}{lll}\label{yt}
\sum\limits_{i=0}^2 \sum\limits_{j=1}^{4-i}\sum\limits_{k=0}^{j-1}{2 \choose i}{4-i\choose j} {j-1\choose k} (-1)^{j+k-1}\left[\left(a^{(i)} ( \phi_2-\phi_1)^{(4-i-j)}\right)^{(j-1-k)}\right]_{x=\xi_0}\delta^{(k)}_{\xi_0} && \\
\nonumber +\sum\limits_{m=0}^1\sum\limits_{i=0}^2 \sum\limits_{j=0}^{1} \sum\limits_{k=0}^{j+i}{2 \choose i}{j+i \choose k} (-1)^{j+i+k}A_{mj}\phi_{2-m}^{(4+j-k)}(\xi_0)\,\,\delta^{(k)}_{\xi_0} && \\
\nonumber +\sum\limits_{m=0}^1 \sum\limits_{i=1}^{2}\sum\limits_{k=0}^{i-1}{2 \choose i} {i-1 \choose k} (-1)^{i-1+k} \left[\left(v_m\phi_{2-m}^{(4-i)}\right)^{(i-1-k)}\right]_{x=\xi_0}\delta^{(k)}_{\xi_0} &&\\
\nonumber +2\sum\limits_{m=0}^1 \left[v_m'\phi''_{2-m}\right]_{x=\xi_0}\delta_{\xi_0} &&\\
-w^2 \sum\limits_{m=0}^1\sum\limits_{j=0}^{3}  \sum\limits_{k=0}^{j} {j \choose k}  (-1)^{j+k} B_{mj}\phi_{_{2-m}}^{(j-k)}(\xi_0)\,\,\delta^{(k)}_{\xi_0}=0 ~. &&
\end{array}
$$
Finally, using the Leibniz rule to expand the derivatives in the first and third lines,
$$
\begin{array}{lll}
\left(a^{(i)} (\phi_2-\phi_1)^{(4-i-j)}\right)^{(j-1-k)}=\sum\limits_{l=0}^{j-1-k} {j-1-k \choose l} a^{(i+j-1-k-l)} ( \phi_2-\phi_1)^{(4-i-j+l)}&&\\
\left(v_m\phi_{2-m}^{(4-i)}\right)^{(i-1-k)}=\sum\limits_{l=0}^{i-1-k} {i-1-k \choose l} v_m^{(i-1-k-l)} \phi_{2-m}^{(4-i+l)} &&
\end{array}
$$
we get, after a long but straightforward calculation:
$$
\begin{array}{lll}
\left[(a_++B_{03}w^2)\phi'''_2-(a_--B_{13}w^2)\phi'''_1+ (a'_+-B_{02}w^2)\phi''_2-\ (a'_-+B_{12}w^2)\phi''_1 \right. &&\\
\left.\hspace{0.4cm}+B_{01}w^2\phi'_2-B_{00}w^2\phi_2+B_{11}w^2\phi'_1-B_{10}w^2\phi_1\right]_{x=\xi_0}\delta_{\xi_0} &&\vspace{.2cm}\\
+\left[(a_+-3B_{03}w^2)\phi''_2-(a_-+3B_{13}w^2)\phi''_1+2B_{02}w^2\phi'_2-B_{01}w^2\phi_2 \right.&&\\
\left.\hspace{0.4cm}+2B_{12}w^2\phi'_1-B_{11}w^2\phi_1 \right]_{x=\xi_0}\delta'_{\xi_0}&&\vspace{.2cm}\\
+\left[a(\phi'_2-\phi'_1)-a'(\phi_2-\phi_1)-A_{01}\phi'''_2-A_{11}\phi'''_1+A_{00}\phi''_2+A_{10}\phi''_1\right.&&\\
\left.\hspace{0.4cm}+3B_{03}w^2\phi'_2+3B_{13}w^2\phi'_1-B_{02}w^2\phi_2-B_{12}w^2\phi_1\right]_{x=\xi_0}\delta''_{\xi_0}&&\vspace{.2cm}\\
+\left[a(\phi_2-\phi_1)+A_{01}\phi''_2- B_{03}w^2\phi_2+A_{11}\phi''_1-B_{13}w^2\phi_1\right]_{x=\xi_0}\delta'''_{\xi_0}=0 \,.
\end{array}
$$
Since $\delta_{\xi_0}$ and its derivatives are linearly independent, this equation is equivalent to the system:
\begin{equation}\label{eeii}
\left\{\begin{array}{ll}
\left[B_{10}w^2\phi_1-B_{11}w^2\phi'_1+\ (a'_-+B_{12}w^2)\phi''_1+(a_--B_{13}w^2)\phi'''_1\right]_{x=\xi_0} \vspace{0.1cm}\\ 
=\left[-B_{00}w^2\phi_2+B_{01}w^2\phi'_2+ (a'_+-B_{02}w^2)\phi''_2+(a_++B_{03}w^2)\phi'''_2\right]_{x=\xi_0}\\
\\
\left[B_{11}w^2\phi_1-2B_{12}w^2\phi'_1+(a_-+3B_{13}w^2)\phi''_1\right]_{x=\xi_0}\vspace{0.1cm}\\
=\left[-B_{01}w^2\phi_2+2B_{02}w^2\phi'_2+(a_+-3B_{03}w^2)\phi''_2\right]_{x=\xi_0} \\
\\
\left[(-a'+B_{12}w^2)\phi_1+(a-3B_{13}w^2)\phi'_1-A_{10}\phi''_1+A_{11}\phi'''_1\right]_{x=\xi_0}\vspace{0.1cm}\\
=\left[(-a'-B_{02}w^2)\phi_2+(a+3B_{03}w^2)\phi'_2+A_{00}\phi''_2-A_{01}\phi'''_2\right]_{x=\xi_0}\\
\\
\left[(a+B_{13}w^2)\phi_1-A_{11}\phi''_1\right]_{x=\xi_0}=\left[(a- B_{03}w^2)\phi_2+A_{01}\phi''_2\right]_{x=\xi_0}
\end{array}\right. 
\end{equation}
which can be rewritten in the form ${\bf A}(\xi_0) \overline{\phi_1(\xi_0)} = {\bf B}(\xi_0) \overline{\phi_2(\xi_0)}$
where ${\bf A}$ and ${\bf B}$ are the matrices (\ref{ee},\,\ref{eee}) and 
$\overline{\phi_i}$ is the column vector $\left( 
  {{\phi_i}},\,{{\phi'_i}},\,{{\phi''_i}},\,{{\phi'''_i}} \right)$, $i=1,2$.
\end{proof}

\begin{corollary}\label{gi-c}
Assuming the conditions of Theorem \ref{jjj}, we have:
\begin{enumerate}
  \item [(i)] If ord $a_i\leq 1$ and ord $b_i \le 3,\,i=0,1$
 then (\ref{12ik},\ref{jjjk},\ref{jjjkkk})  has continuous solutions. 
  \item [(ii)]
If ord $a_i=0$ and ord $b_i \le 2,\,i=0,1$ then (\ref{12ik},\ref{jjjk},\ref{jjjkkk}) has continuously differentiable solutions.
\end{enumerate}
\end{corollary}

\begin{proof}

(i) It follows from the fourth equation in the system (\ref{eeii}) that if $A_{i1}=0$ (i.e. ord $a_i\leq1$), $i=0,1$; and $ B_{i3}=0$ (i.e. ord $b_i\leq3$), $i=0,1$ then 
\begin{equation}\label{gi}
\phi_2 (\xi_0)=\phi_1 (\xi_0) ~,
\end{equation}
since $a(\xi_0)\not= 0$ (because (\ref{12ik}) satisfies (C1'')). Notice that there are other conditions for which the solutions of (\ref{12ik}) are continuous (e.g. $A_{11}=A_{01}=0$ and $B_{13}=-B_{03}\not=-a(\xi_0)/w^2 $). We only stated the simplest case.

 

(ii) From the third and fourth equations in (\ref{eeii}) we easily conclude that if $A_{ij}=0$ and $B_{i3}=B_{i2}=0,\,i,j=0,1$, (i.e. ord $a_i=0$ and ord $b_i\leq 2,\,i=0,1$)  then 
\begin{equation}\label{gii}
\phi_2 (\xi_0)=\phi_1 (\xi_0) \text{ and } \phi'_2 (\xi_0)=\phi'_1 (\xi_0)~.
\end{equation}
It then follows from (\ref{Solution1}) that the solutions of (\ref{12ik}) are continuously differen--tiable on $\RE$. Once again, there are other, more involved, conditions for which this property also holds. 
\end{proof}


From Theorem \ref{jjj} we also find that:

\begin{corollary} \label {hh}
Consider the ODE (\ref{12ik}), satisfying the conditions of Theorem \ref{jjj}, with arbitrary boundary conditions $\overline{\phi (x_0)}= \overline{C}$ at $x_0 < \xi_0$ (respectively $x_0 > \xi_0$), where $\overline{C} \in \CO^4$.
The solution of this BVP exists and is unique  if and only if $\det {\bf B}(\xi_0) \not=0$ (respectively $\det {\bf A}(\xi_0) \not=0$), where ${\bf A}$ and ${\bf B}$ are the matrices given by (\ref{ee},\,\ref {eee}) and $\overline{\phi}=\left({\phi},\,{{\phi'}},\,{{\phi''}},\,{{\phi'''}} \right)^T $.
\end{corollary}

\begin{proof}
It follows from Theorem \ref{jjj} that the solutions of the ODE (\ref{12ik}) satisfy the interface conditions  (\ref{ICond}) with ${\bf A}$ and ${\bf B}$ given by (\ref{ee},\,\ref {eee}).

Moreover, in this case, (\ref{12ik}) also satisfies the pre-requisites of Theorem \ref{1ii} with 
$I=\{\xi_0\}$ and $M\le 4$. It then follows from this Theorem that the solution of the BVP (\ref{12ik},\,\ref{IC}) for $x_0 < \xi_0$  
exists and is unique if and only if $\det {\bf B}(\xi_0) \not=0$. An equivalent result holds for the case $x_0 > \xi_0$, concluding the proof. 
\end{proof}

Some specific cases for which the solution of the BVP (\ref{12ik},\,\ref{IC}) exists and is unique are the following:

\begin{corollary}
Consider the  ODE (\ref{12ik}) with coefficients (\ref{jjjk},\,\ref{jjjkkk}) satisfying the conditions of Theorem \ref{jjj}. 
For $i=0$ (or $i=1$) assume that the coefficients $a_i, b_i$ satisfy one of the following conditions:
\begin{itemize}
\item [(a)] ord $a_i=0$ and ord $b_i \le 3$,
 \item [(b)] ord $a_i\le 1$ and ord $b_i\leq 2$
\item [(c)] ord $a_i \le 2$ and ord $b_i \le 1$
\end{itemize}
Then the solution of (\ref{12ik}) exists and is unique for arbitrary boundary condi--tions
$
\overline{\phi(x_0)} =\overline{C}
$ 
given at $x_0 < \xi_0$ (respectively, $x_0 > \xi_0$) .



\end{corollary}

\begin{proof}
This Corollary follows directly from the previous one. 
The conditions (a)--(c) are equivalent to:
\begin{itemize}
\item [(a')] $A_{ij}=0,\, j=0,1$ and $B_{i3}=0$
\item [(b')] $A_{i1}=0$ and $B_{i2}=B_{i3}=0$
\item [(c')] $B_{i1}=B_{i2}=B_{i3}=0$
\end{itemize}
For $i=0$ any of these conditions implies $\det \,{\bf B}(\xi_0)=a_+^2(\xi_0) a^2(\xi_0) \not=0$ (the inequality follows from the conditions (C1'') and (C2'') of Theorem \ref{jjj}). From the previous Corollary we conclude that the solution of the BVP exists and is unique for arbitrary boundary conditions given at $x_0 < \xi_0$. An equivalent statement holds for $i=1$. In this case $\det \,{\bf A}(\xi_0)=a_-^2(\xi_0) a^2(\xi_0) \not=0$.


\end{proof}

\section {Vibration modes of the Euler--Bernoulli beam with a crack}

In this section we use the previous results to study the vibration modes of an EBB with discontinuous flexural stiffness and a structural crack. One interesting feature of our formulation is that it is well-defined even if the crack is located at points where the flexural stiffness is discontinuous.   

\subsection{The model}

We consider the EBB boundary value problem:
\begin{equation}\label{hh7-0}
\left[a(x) u''(x,t)\right]'' + m\ddot{u}(x,t)=0,\quad x\in\left[0,1\right],\quad t\geq0
\end{equation}
with pinned-pinned (PP) boundary conditions (\ref{nnn})
or clamped-clamped (CC) boundary conditions (\ref{CCbc}).
Here $u$ is the transverse displacement, $m$ is the (constant) mass per unit length and 
\begin{equation} \label{ll-0}
a(x)  =  A( H (\xi_0-x) -\lambda_0 \delta_{\xi_0})+ 
 kA( H (x-\xi_0) -\lambda_1\delta_{\xi_0}) 
\end{equation}
where $A >0$ and $kA>0$ are the flexural stiffness  constants  in the sections $[0,\xi_0)$ and $(\xi_0,1]$, respectively ($0< \xi_0<1$). Moreover, $\lambda_0,\lambda_1$ are dimension--less parameters related to the intensity of the crack at the left and right sides of the junction point. Following \cite{BC07,Cad09,Cad12}, the crack at $x=\xi_0$ is modelled by a flexural stiffness displaying a Dirac delta function at $\xi_0$. If the beam is uniform (k=1) the crack is modelled by $A (\lambda_0+\lambda_1) \delta_{\xi_0}$ and its intensity (and all subsequent results) are dependent of $\lambda_0+\lambda_1$ only, and not of the individual values of $\lambda_0$ and $\lambda_1$. An interesting question is whether this is also the case if the beam is not uniform.

Equation (\ref{hh7-0}) with coefficients (\ref{ll-0}) does not in general display smooth solutions, and is not well-defined for non-smooth ones either. Following the approach of the previous section, we formulate the system using the more general equation (\ref{21}):
\begin{equation} \label{hh7}
\left[a_0(x)* u''(x,t)+ u''(x,t)*a_1(x)\right]'' + m\ddot{u}(x,t)=0,\,x\in\left[0,1\right],\, t\geq 0
\end{equation}
where the coefficients $a_0,a_1$ should satisfy $a_0+a_1=a$ (cf.(\ref{ll-0})) while modelling the correct inner structure of the crack, i.e. with intensities $\lambda_0,\lambda_1$ at $\xi_0^-$ and $\xi_0^+$, respectively. Using (\ref{LDE1}) as a guideline, we set: 
\begin{eqnarray}\label{ll-novo}
a_0(x) & = & A H(\xi_0-x) -kA\lambda_1 \delta_{\xi_0} \\
a_1(x) &=& kA H(x-\xi_0) -A\lambda_0\delta_{\xi_0} \nonumber
\end{eqnarray}
so that in (\ref{LDE1}) we have:
\begin{eqnarray}
a_{0\epsilon}(x-\epsilon) + a_{1\epsilon}(x+\epsilon) &=& A \left[ H^\epsilon((\xi_0+\epsilon)-x) - \lambda_0\delta^\epsilon_{\xi_0-\epsilon}  \right] \nonumber \\
&& + kA \left[ H^\epsilon(x-(\xi_0-\epsilon)) - \lambda_1\delta^\epsilon_{\xi_0+\epsilon} \right] \nonumber 
\end{eqnarray}
where $H^\epsilon$ and $\delta_a^\epsilon$ are suitable one-parameter families of smooth functions converging in $\DO'$ (as $\epsilon \to 0^+$) to $H$ and $\delta_a$, respectively (cf. \cite{DJP20}). We also have
$
\lim_{\epsilon \to 0^+} 
a_{0\epsilon}(x-\epsilon) + a_{1\epsilon}(x+\epsilon)=a(x) \, 
$ 
in $\DO'$, as it should. 

There are, of course, other choices of the coefficients $a_0$ and $a_1$, still satisfying $a_0+a_1=a$, and which can also be suitable to model this specific system. Our particular choice aims at yielding the simplest formulation. We will discuss other possibilities in section 4.5.

The separable solutions of (\ref{hh7}) are of the form:
\begin{equation} \label{Ssolutionsn}
u_n(x,t)=\phi_n(x)T_n(t)\quad , \quad n=1,2,..
\end{equation}
and for PP (\ref{nnn}) or CC (\ref{CCbc}) boundary conditions they form a discrete set. Substituting (\ref{Ssolutionsn}) in (\ref{hh7}), we easily obtain the eigenvalue equation for the vibration modes of the beam: 
\begin{equation}\label{78p} 
\left[( H (\xi_0-x) -k\lambda_1 \delta_{\xi_0})* \phi_n''(x) + \phi_n''(x) *( k H (x-\xi_0) -\lambda_0\delta_{\xi_0})\right]''-\alpha_n^4 \phi_n(x)=0
\end{equation}
where $\alpha_n$ is a frequency parameter given by
\begin{equation}\label{787}
\alpha_n^4=\frac{w_n^2 m}{A}
\end{equation}
and $w_n$ is the natural frequency of the $n$th vibration mode.

\subsection{General solutions and interface conditions}

Equation (\ref{78p}) is of the general form (\ref{12ii}). Moreover, the coefficients $a_0,a_1$ satisfy the conditions (C1'') and (C2'') of Theorem \ref{jjj}. Hence, its solutions are of the  form (cf. eq.(\ref{Solution1})):
\begin{equation} \label{789}
\phi_n (x)=H (\xi_0-x) \phi_{1n} (x) + H (x-\xi_0)\phi_{2n} (x)
\end{equation}
where the vibration modes at the left and right sides of the crack ($\phi_{1n}$ and $\phi_{2n}$, respectively) satisfy (cf. (\ref{ODESReg}), and note from (\ref{78p}) that $a_-=1,\, a_+=k,\, b_\pm=m/A$):
\begin{equation}\label{8p}
\phi_{1n}^{(4)}(x)-\alpha_n^4 \phi_{1n} (x)=0 \qquad \mbox{and} \qquad k\phi_{2n}^{(4)}(x)-\alpha_n^4 \phi_{2n} (x)=0
\end{equation}
on $[0,\xi_0]$ and $[\xi_0,1]$, respectively. 


The solutions of (\ref {8p}) are easily found by first solving the characteristic equations:
$$
s_1^4-\alpha_n^4=0 \qquad,\qquad k s_2^4-\alpha_n^4=0\, ,
$$
the roots of which are 
$$
s_1=\pm\alpha_n,\quad s_1=\pm i\alpha_n\qquad \mbox{and} \qquad s_2=\pm\frac{\alpha_n}{\sqrt[4]{k}},\quad s_2=\pm i\frac{\alpha_n}{\sqrt[4]{k}}\, , 
$$
and thus:
\begin{equation}\label{v1}
\phi_{1n} (x)=A_1\sin(\alpha_n x)+B_1\cos (\alpha_n x)+C_1\sinh (\alpha_n x)+D_1\cosh (\alpha_n x)
\end{equation}
and
\begin{equation}\label{v2}
\phi_{2n}(x)=A_2\sin(\alpha_n\beta x)+B_2\cos (\alpha_n\beta x)+C_2\sinh (\alpha_n\beta x)+D_2\cosh (\alpha_n\beta x)
\end{equation}
where $A_i,B_i,C_i,D_i,\, i=1,2$ are integration constants (which may vary with $n$), $\alpha_n$ was defined in (\ref{787}) and
\begin{equation}\label{ggj}
\beta= k^{-1/4}~.
\end{equation}

It also follows from Theorem \ref{jjj} that the solutions (\ref{789}) satisfy an interface condition at $x=\xi_0$ (cf. eqs.(\ref{ICond},\,\ref{ee},\,\ref{eee})):
$${\bf A}(\xi_0) \overline{\phi_{1n}(\xi_0)} = {\bf B} (\xi_0)\overline{\phi_{2n}(\xi_0)}$$
where 
$
\overline{\phi_{in}}= \left( 
  {{\phi_{in}}},\,{{\phi'_{in}}},\,{{\phi''_{in}}},\,{{\phi'''_{in}}} \right)^T, \, i=1,2;
$
and the matrices ${\bf A}(\xi_0)$, ${\bf B}(\xi_0)$ are given explicitly by:
\begin{equation*} 
{\bf A}(\xi_0)=\left( {\begin{array}{*{20}{c}}
  0&0&0&1\\ 
  0&0&1&0\\ 
	0&1+k&\lambda_0&0\\ 
  1+k&0&0&0
\end{array}} \right),\;  {\bf B}(\xi_0)=\left( {\begin{array}{*{20}{c}}
  0&0&0&k\\ 
  0&0&k&0\\ 
	0&1+k&-k\lambda_1&0\\ 
  1+k&0&0&0
\end{array}} \right) ~.
\end{equation*}
Note that, from (\ref{78p}), we have: $a_{0-}=1$, $a_{0+}=0$, $a_{1-}=0$, $a_{1+}=k$, $A_{00}=-k\lambda_1$, $A_{10}=-\lambda_0$, $A_{i1}=B_{ij}=0,\,i=0,1;\,j=0,1,2,3$.

The interface condition can be easily re-written in the form:
\begin{equation} \label{yu}
\left\{\begin{array}{ll}
   \phi'''_{1n} (\xi_0)-k\phi'''_{2n} (\xi_0)=0\\
   \phi''_{1n} (\xi_0)-k\phi''_{2n}  (\xi_0)=0\\
   \phi'_{1n} (\xi_0)+S\phi''_{1n}  (\xi_0)- \phi'_{2n}  (\xi_0)=0\\
   \phi_{1n} (\xi_0)-\phi_{2n} (\xi_0)=0 
\end{array}\right. 
\end{equation} 
where
\begin{equation}\label{uy}
S=\frac{\lambda_0+\lambda_1}{k+1} ~.
\end{equation} 
We remark that (\ref{yu}\,,\ref{uy}) are only dependent of the value of $\lambda_0+\lambda_1$ and not of the individual values of $\lambda_0,\lambda_1$. Hence, in this model, only the total intensity of the crack affects the dynamical behaviour of the beam. We will see, however, that this is not the case for other choices of the coefficients $a_0,a_1$.

Finally, and since $\det {\bf A}(\xi_0)=(1+k)^2 \not=0$ and also $\det {\bf B}(\xi_0)=k^2 (1+k)^2 \not=0$, it follows from Corollary \ref{hh} that the solution of eq.(\ref{78p}) exists and is unique for arbitrary boundary conditions $\overline{\phi (x_0)}= \overline{C}$ at $x_0 \not= \xi_0$. To be more specific, this means that the four boundary conditions $\overline{\phi (x_0)}= \overline{C}$ and the four interface conditions (\ref{yu}) completely fix the eight integration constants in eqs.(\ref{v1},\,\ref{v2}). 


\subsection {Boundary conditions}

We consider two cases: pinned-pinned (PP) boundary conditions and clamped-clamped (CC) boundary conditions.

\subsubsection{Simply supported  (pinned-pinned) beam}\label{subsubsec:pp}

The imposition of PP boun--dary conditions (\ref{nnn}) on the solutions (\ref{789}) yields:
\begin{equation} \label{17}
\phi_{1n}(0)=0; \quad \phi''_{1n}(0)=0
\end{equation}
\begin{equation} \label{18}
\phi_{2n}(1)=0;\quad \phi''_{2n}(1)=0 ~.
\end{equation}
Substituting (\ref{17}) in (\ref{v1}), we get:
\begin{equation} \label{199}
B_1=D_1=0 ~.
\end{equation}
If we impose the boundary conditions (\ref{18}) and the interface conditions (\ref{yu}) on (\ref{v1},\ref{v2}) and take into account (\ref{199}), we get a linear and homogeneous system of six equations for the six unknowns $A_1,C_1,A_2,B_2,C_2$ and $D_2$. This system can be written as:
\begin{equation}\label{CondPP}
{\bf M}_{PP} \, {\bf X}=0
\end{equation}
where 
$
{\bf X}=(A_1,C_1,A_2,B_2,C_2,D_2)^T
$ and
\begin{equation} \label{19}
{\bf M}_{PP}= \left( {\begin{array}{*{20}{c}}
  0&0&{{s_1}}&{{c_1}}&{s{h_1}}&{c{h_1}} \\ 
  0&0&{ - {s_1}}&{ - {c_1}}&{s{h_1}}&{c{h_1}} \\ 
  {{s_3}}&{s{h_3}}&{ - {s_2}}&{ - {c_2}}&{ - s{h_2}}&{ - c{h_2}} \\ 
   {f_1}&{f_2}&\beta{c_2}&{ - \beta{s_2}}&\beta c{h_2}&{ \beta s{h_2}} \\ 
  { - {s_3}}&{s{h_3}}&{\sqrt[]{k}{s_2}}&{  \sqrt[]{k}{c_2}}&{ - \sqrt[]{k}s{h_2}}&{ - \sqrt[]{k}c{h_2}} \\ 
  { - {c_3}}&{c{h_3}}&{\sqrt[4]{k}{c_2}}&{ - \sqrt[4]{k}{s_2}}&{ - \sqrt[4]{k}c{h_2}}&{ - \sqrt[4]{k}s{h_2}} 
\end{array}} \right) 
\end{equation}
Here, 
\begin{eqnarray}\label{gg}
&& s_i=\sin(\xi_{0i}\alpha_n\beta_i),\quad c_i=\cos(\xi_{0i}\alpha_n\beta_i)\\
\nonumber &&sh_i=\sinh(\xi_{0i}\alpha_n\beta_i),\quad ch_i=\cosh(\xi_{0i}\alpha_n\beta_i),\quad i=1,2,3
\end{eqnarray}
where $\xi_{01}=1$, $\xi_{0i}=\xi_0$, $i=2,3$, $\beta_{3}=1$, $\beta_{i}=\beta$, $i=1,2$.
 Moreover:
\begin{equation} \label{SS}
f_1=- {c_3} + \alpha_n S {s_3},\,\,\,f_2=- c{h_3} - \alpha_n S{sh_3}
\end{equation}
and $\beta$, $S$ are defined in (\ref{ggj}) and (\ref{uy}), respectively.

The homogeneous system (\ref{CondPP}) has nontrivial solutions, if and only if:
\begin{equation}\label{detEq}
\det {\bf M}_{PP} =0
\end{equation}
yielding the characteristic equation for the frequency parameters $\alpha_n$. Solving for $\alpha_n$ and substituting back into (\ref{CondPP}), one can then obtain the integration constants, and thus the explicit form of the vibration modes. Finally, we notice that the frequency $\alpha_n$, solution of \eqref{detEq}, is in general dependent of $\xi_0,k$ and $S$, but not of $\lambda_1$ and $\lambda_2$ individually. We will study these relations numerically.


\subsubsection{ Clamped-clamped beam (CC)} \label{subsubsec:cc}

Following the same steps as for the previous case, we impose the CC boundary conditions (\ref{CCbc}) on the solutions (\ref{789}) and get:
\begin{equation} \label{23}
\phi_{1n}(0)=0; \quad \phi'_{1n}(0)=0
\end{equation}
\begin{equation} \label{24}
\phi_{2n}(1)=0;\quad \phi'_{2n}(1)=0
\end{equation}

Substituting the conditions (\ref{23}) in (\ref{v1}), yields:
\begin{equation} \label{233}
A_1=-C_1 \text{ and } B_1=-D_1
\end{equation}
As before, if we impose the boundary conditions (\ref{24}) and the interface conditions (\ref{yu}) on (\ref{v1},\ref{v2}) and take into account (\ref{233}), we get a linear and homogeneous system of six equations for the six unknowns $C_1,D_1,A_2,B_2,$ $C_2,D_2$. This system can be written as:
\begin{equation}\label{CondCC}
{\bf M}_{CC} \, {\bf X}=0
\end{equation}
where 
$
{\bf X}=(C_1,D_1,A_2,B_2,C_2,D_2)^T ~,
$
\begin{equation} \label{25}
{\bf M}_{CC}= \left( {\begin{array}{*{20}{c}}
  0&0&{{c_1}}&{{-s_1}}&{c{h_1}}&{s{h_1}} \\ 
  0&0&{  {s_1}}&{  {c_1}}&{s{h_1}}&{c{h_1}} \\ 
  {{-s_3+ sh_3}}&{-c_3 + ch_3}&{ - {s_2}}&{ - {c_2}}&{ - s{h_2}}&{ - c{h_2}} \\ 
  { g_1}&{ g_2}&{-\beta{c_2}}&{  \beta{s_2}}&{-\beta c{h_2}}&{-\beta s{h_2}} \\ 
  {  {s_3+sh_3}}&{c_3+ch_3}&{\sqrt{k}{s_2}}&{  \sqrt{k}{c_2}}&{ - \sqrt{k}s{h_2}}&{ - \sqrt{k}c{h_2}} \\ 
  { c_3+ch_3}&{-s_3+sh_3}&{\sqrt[4]{k}{c_2}}&{ - \sqrt[4]{k}{s_2}}&{ - \sqrt[4]{k}c{h_2}}&{ - \sqrt[4]{k}s{h_2}} 
\end{array}} \right)~,
\end{equation}
$s_i,c_i,sh_i,ch_i,\,\,i=1,2,3$ are defined in (\ref{gg}), and
\begin{eqnarray*}\label{a}
g_1=- {c_3}+ch_3 + S\alpha_n (s_3+sh_3),\,\,\,\,\,\,g_2=s_3+s{h_3} + S \alpha_n (c_3+ch_3) ~.
\end{eqnarray*} 
The homogeneous system (\ref{CondCC}) has nontrivial solutions, if and only if:
\begin{equation}\label{detEqh}
\det {\bf M}_{CC} =0
\end{equation}
yielding the characteristic equation for the frequency parameters $\alpha_n$. As for the case of PP boundary conditions, if we substitute these frequencies back into (\ref{CondCC}), we can obtain the integration constants, and thus the explicit form of the vibration modes.

\subsection{Numerical results}
We now solve (\ref{detEq}) and (\ref{detEqh}) numerically and determine the values of the frequency parameters $\alpha_n$ for the PP and CC beams. The results for several different cases are displayed in Figures 1 to 3 below. 




Figure 1 displays the first, second and third frequencies $ \alpha_i, i=1,2,3$ of the CC and PP (uniform and non-uniform) beams as a function of the crack intensity parameter $\lambda = \lambda_0=\lambda_1 $ (the intensities of the crack at the left and right sides of the junction point are assumed identical). The junction point and the crack are both localized at the midpoint of the beam.
It is interesting to note that for the CC and PP beams (uniform and non-uniform) the second frequency is independent of the value of $ \lambda$.

\begin{figure}[h]
\center\includegraphics [scale=0.6] {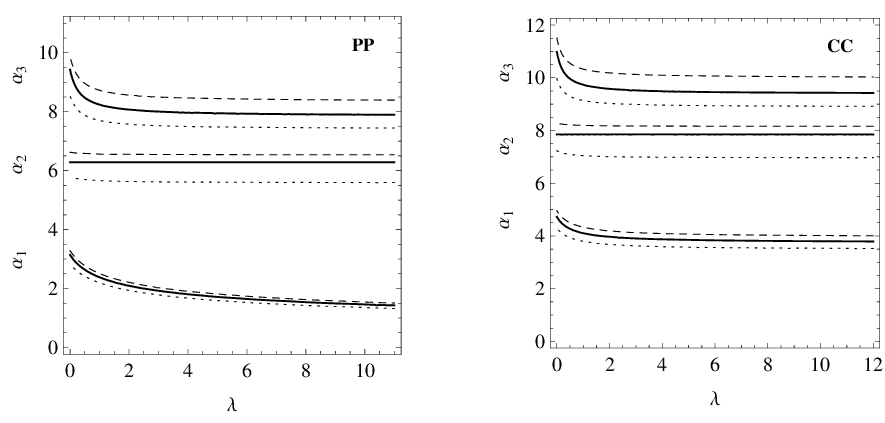}
\end{figure}
\footnotesize {Fig.1: 1st, 2nd and 3rd frequencies of the CC and PP (uniform and non-uniform) beams versus the intensity $\lambda=\lambda_0=\lambda_1$ of a crack localized (together with a junction point) at the midpoint of the beam. The uniform cases ($k = 1$) are represented by thick lines, the non-uniform beams by dotted lines (for $k = 0.5 $) and dashed lines (for $k = 1.5 $).} 
\vspace{0.3cm}

\normalsize

Figure 2 shows the first, second and third frequencies of the CC and PP beams as a function of $k$. A single crack of three possible intensities, and a junction point are both localized at the midpoint of the beam ($\xi_0=1/2$).





\begin{figure}[h]
\center\includegraphics [scale=0.6]{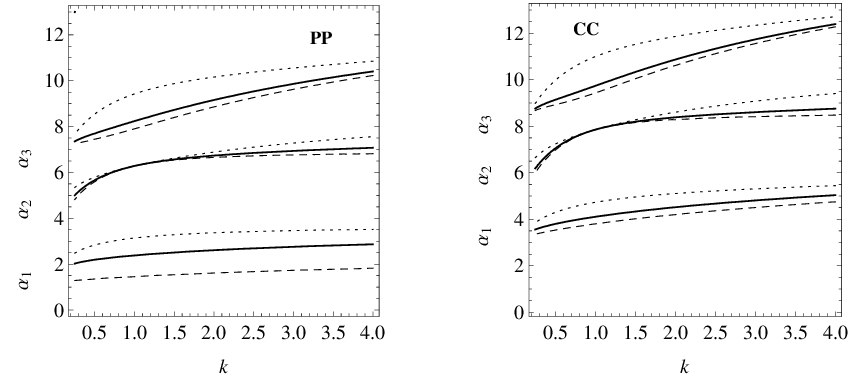}
\end{figure}
\footnotesize {Fig.2: 1st, 2nd and 3rd frequencies of the CC and PP beams versus $k$. The beams have a single crack, of intensity $\lambda=\lambda_0=\lambda_1$, which is localized at the junction point in the midpoint of the beam ($\xi_0=1/2$). Three cases are considered: $\lambda=0$ (dotted lines), $\lambda=1$ (thick lines) and $\lambda=10$ (dashed lines).} 
\vspace{0.3cm}

\normalsize

Figure 3 displays the first, second and third frequencies of the CC and PP (uniform and non-uniform) beams as a function of the crack position $\xi_0$. In all cases the crack intensity is $\lambda_0=\lambda_1=2$.

\begin{figure}[h]
\center\includegraphics [scale=0.7]{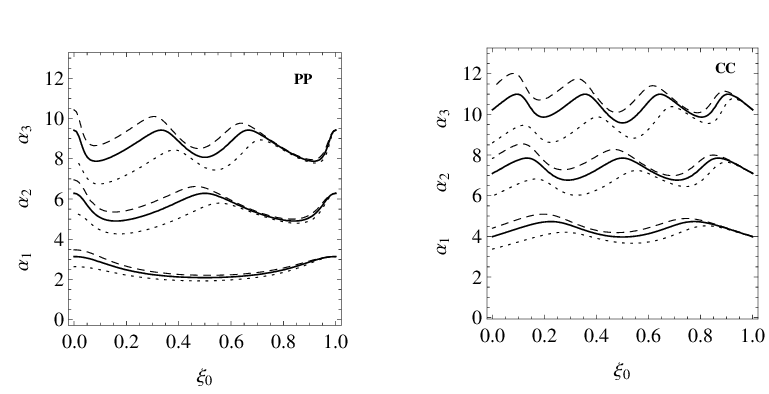}
\end{figure}
\footnotesize {Fig.3: 1st, 2nd and 3rd frequencies of the CC and PP (uniform and non-uniform) beams versus the crack position (which again coincides with the junction point). The uniform beams ($ k = 1 $) are represented by the thick lines and the non-uniform beams by dotted lines ($k = 0.5 $) and dashed lines ($k = 1.5$). The crack intensity is $\lambda_0=\lambda_1=2$.} 
\vspace{0.3cm}

\normalsize

\subsection{Discussion and outlook}

The graphs of the 1st frequency for uniform beams shown in Figures 1 and 3 are remarkably similar to the ones given in Figures 1 and 2a of \cite{Cad09} (the other graphs concern different systems). Note that the damage parameter $\lambda$ that is used in \cite{Cad09} is related to the crack intensity $\gamma$ by (we are using the notation of \cite{Cad09}, cf.[eq.(C19), \cite{Cad09}])
$$
\lambda= \frac{\gamma}{1-A \gamma},
$$ 
and so, in our case, it satisfies $\lambda=\gamma$. This is because $A$ is a parameter related to the product of two Dirac deltas with the same support (cf. [eq.(C18), \cite{Cad09}]), which for our product $*$ is given by $\delta_{\xi_0} *\delta_{\xi_0} =0 \Longrightarrow A=0$ (in \cite{Cad09} the authors used the Bagarello product \cite{Bag95,Bag02}). 

Another interesting feature of the model (\ref{hh7},~\ref{ll-novo}) is that the characteristic frequencies of the beam are only dependent of the total intensity of the crack and not of its inner structure (i.e. of the particular left and right intensities at the junction point). This is a consequence of our choice of coefficients (\ref{ll-novo}), and not a general property of the model \eqref{hh7}.   
 If we consider a more general crack inner structure, by setting:
$$
a_i(x)=\tfrac{1}{2}\left[ AH(\xi_0-x) + kA H(x-\xi_0)\right] -A \lambda_{i0} \delta_{\xi_0}-kA \lambda_{i1} \delta_{\xi_0} \quad i=0,1
$$
where the new parameters $\lambda_{ij}$, $i,j=0,1$ satisfy $\lambda_{00}+\lambda_{10}=\lambda_0$ and $\lambda_{01}+\lambda_{11}=\lambda_1$, then the relation $a_0(x)+a_1(x)=a(x)$ still holds (cf.(\ref{ll-0})). However, it is now possible to choose the parameters $\lambda_{ij}$ in such a way that the specific left and right split of the crack intensity does affect the characteristic frequencies. This model is more general, and may be interesting for fine-tuning the description of beams with cracks at junction points.








\subsection*{Declarations}

\subsubsection*{Funding}

The work of Cristina Jorge was supported
by the PhD grant SFRH/BD/85839/2016 of the Portuguese Science
Foundation.

\subsubsection*{Competing interests}

The authors have no relevant financial or non-financial interests to disclose.

\subsubsection*{Data availability statement} 

Data sharing not applicable to this article as no datasets were generated or analysed during the current study.

\vspace{0cm}

***************************************************************

\textbf{Author's addresses:}

\begin{itemize}
\item \textbf{Nuno Costa Dias}{\footnote{Corresponding author}} and \textbf{Jo\~ao Nuno Prata:
}Escola Superior N\'autica Infante D. Henrique. Av. Eng.
Bonneville Franco, 2770-058 Pa\c{c}o d'Arcos, Portugal and Grupo
de F\'{\i}sica Matem\'{a}tica, Universidade de Lisboa, Av. Prof.
Gama Pinto 2, 1649-003 Lisboa, Portugal

\item \textbf{Cristina Jorge}: Departamento de Matem\'{a}tica.
Universidade Lus\'{o}--fona de Humanidades e Tecnologias. Av. Campo
Grande, 376, 1749-024 Lisboa, Portugal and Grupo de F\'{\i}sica
Matem\'{a}tica, Universidade de Lisboa, Av. Prof. Gama Pinto 2,
1649-003 Lisboa, Portugal

\end{itemize}

\vspace{-0.1cm}

\small

{\it E-mail address} (NCD): ncdias@meo.pt

{\it E-mail address} (CJ): cristina.goncalves.jorge@gmail.com

{\it E-mail address} (JNP): joao.prata@mail.telepac.pt

\end{document}